\documentclass[11pt]{amsart}
 
\usepackage{hyperref, color}
\newtheorem{theorem}{Theorem}[section]
\newtheorem{lemma}[theorem]{Lemma}
\newtheorem{proposition}[theorem]{Proposition}

\newtheorem{assumption}[theorem]{Assumption}

\usepackage{esint} % for the \fint command
\usepackage{amsmath}

\providecommand{\Ric}{\operatorname{Ric}}
\providecommand{\Sec}{\operatorname{Sec}}

\providecommand{\tr}{\operatorname{tr}}

\newcommand{\II}{\mathrm I\!\mathrm I}

\theoremstyle{definition}
\newtheorem{definition}[theorem]{Definition}

\theoremstyle{remark}

\numberwithin{equation}{section}

\begin{document}

\title
[Rigidity Results for CMC Hypersurfaces]
{First Eigenvalue of Jacobi Operator and Rigidity Results for Constant Mean Curvature Hypersurfaces}

\author[Batista]{M\'arcio Batista}
\author[Cavalcante]{Marcos P. Cavalcante}
\author[Melo]{Luiz R. Melo}

\address{CPMAT-IM, Universidade Federal de Alagoas, Macei\'o - Brazil}
\email{mhbs@mat.ufal.br, marcos@pos.mat.ufal.br, luiz.melo@im.ufal.br}

\date{\today}

\commby{}

\begin{abstract}
In this paper, we obtain geometric upper bounds for the first eigenvalue $\lambda_1(J)$ of the Jacobi operator for both closed hypersurfaces and compact hypersurfaces with boundary having constant mean curvature (CMC). As an application, we derive new rigidity results for the area of CMC hypersurfaces under suitable conditions on $\lambda_1(J)$ and the curvature of the ambient space. We also address the Jacobi--Steklov problem, proving geometric upper bounds for its first eigenvalue $\sigma_1(J)$ and deriving rigidity results related to the length of the boundary. Additionally, we present some results in higher dimensions related to the Yamabe invariants.
\end{abstract}

\maketitle

%=========================================================
\section{Introduction}
%=========================================================
In the theory of minimal and constant mean curvature hypersurfaces, the Jacobi operator is one of the basic tools connecting spectral data with geometry and topology. Its first eigenvalue often detects strong rigidity phenomena, especially under lower curvature bounds on the ambient manifold.

For a two-sided compact hypersurface $\Sigma^n$ immersed in an orientable Riemannian manifold $M^{n+1}$, the Jacobi operator is
\[
J=\Delta+\Ric(N,N)+|A|^2,
\]
where $N$ is a global unit normal along $\Sigma^n$ and $A$ is the shape operator. The first eigenvalue of $J$, denoted by $\lambda_1(J)$, plays a central role in the study of stability and in the interaction between the geometry of $\Sigma^n$ and the curvature of the ambient manifold.

In dimension two, the interplay between $\lambda_1(J)$, the scalar curvature of the ambient manifold, and the topology of the surface goes back to the work of Schoen and Yau \cite{SchoenYau}, who showed that a closed orientable stable minimal surface in a $3$-manifold with nonnegative scalar curvature must be either a sphere or a torus, with infinitesimal rigidity in the torus case. Since then, several estimates for $\lambda_1(J)$ have been obtained for closed hypersurfaces and surfaces under different geometric assumptions; see, for instance, \cite{ABB05, AMO15, CC17, Cheng08, CLS17, DPT23, MO14, perdomo2002, Zhu}. We also mention the work of the first-named author with Santos \cite{BatistaSantos}, where general upper bounds for $\lambda_1(J)$ were obtained for closed CMC surfaces in $3$-manifolds under a lower bound on the scalar curvature.

The purpose of this paper is twofold. We first establish a general dimension-$n$ estimate for the first eigenvalue of the Jacobi operator in the closed case and then specialize it to dimension two, where Gauss--Bonnet converts the scalar curvature term into topological information. After that, we turn to the free-boundary setting and to the associated Jacobi--Steklov problem, again proceeding from the general dimension-$n$ framework to the two-dimensional consequences.

Throughout the paper, unless otherwise stated, we assume that $M^{n+1}$ and $\Sigma^n$ are connected and orientable, and that $\Sigma^n$ is compact. When $\partial\Sigma\neq\emptyset$, we assume that $\Sigma^n$ has free boundary in $M^{n+1}$ in the sense that $\partial\Sigma$ intersects $\partial M$ orthogonally.

Before stating our results, we fix some notation. We denote by $\Sec$, $\Ric$, and $S$ the sectional, Ricci, and scalar curvatures of $M^{n+1}$, respectively. We denote by $S_\Sigma$ the scalar curvature of the induced metric on $\Sigma^n$. If $\partial M\neq\emptyset$, then $\II^{\partial M}$ denotes the second fundamental form of $\partial M$ in $M$, computed with respect to the outward unit normal.

The statement below is the main result of the first part of the paper.

\begin{theorem}\label{T4}
Let $M^{n+1}$ be a closed orientable Riemannian manifold with scalar curvature $S \geq n(n+1)$, and let $\Sigma^n \subset M^{n+1}$ be a closed, immersed, orientable hypersurface with constant mean curvature $H$. If $\lambda_1(J) \geq -n$, then
\begin{equation}\label{media}
n(n-1)|\Sigma^n| \leq \int_{\Sigma^n} S_{\Sigma} \, dv.
\end{equation}

Moreover, equality holds in \eqref{media} if and only if $\Sigma^n$ is totally geodesic, $S\big|_{\Sigma} = n(n+1)$, and $\Ric(N,N)=n$.

Furthermore, in the equality case in \eqref{media}, if $\Sec\geq 1$ and either
\begin{enumerate}
    \item $n$ is even, or
    \item $n$ is odd and $\Sigma^n$ is simply connected,
\end{enumerate}
then $M^{n+1}$ is isometric to $\mathbb{S}^{n+1}$ with its canonical metric.
\end{theorem}

Now, we specialize to the case $n=2$. In this dimension, the scalar curvature term on $\Sigma^2$ can be converted into topological information via Gauss--Bonnet, leading to a clean bridge from the general analytic identities to geometric rigidity statements for surfaces.

For closed surfaces in ambient manifolds with nonnegative scalar curvature, we prove the following theorem.

\begin{theorem}\label{T2}
Let $M^3$ be a Riemannian manifold with scalar curvature $S \geq 0$.
Let $\Sigma^2 \subset M^3$ be a closed immersed orientable surface with constant mean curvature $H \geq 2$ and $\lambda_1(J) \geq -2$. Then $\Sigma^2$ has genus $0$ and $|\Sigma^2| \leq 4\pi$.

Furthermore, if equality holds in the area inequality and $\Sigma^2$ bounds a domain, then $\Sigma^2$ is isometric to the unit sphere, and $M^3$ contains a subset isometric to the Euclidean unit ball whose boundary is $\Sigma^2$.
\end{theorem}

We also obtain a hyperbolic counterpart.

\begin{theorem}\label{T3}
Let $M^3$ be a Riemannian manifold with scalar curvature $S \geq -6$. Let $\Sigma^2 \subset M^3$ be a closed immersed orientable surface with constant mean curvature $H \geq 2\sqrt{2}$ and $\lambda_1(J) \geq -2$. Then $\Sigma^2$ has genus $0$ and $|\Sigma^2| \leq 4\pi$.

Furthermore, if equality holds in the area inequality and $\Sigma^2$ bounds a domain, then $\Sigma^2$ is isometric to a sphere, and $M^3$ contains a subset isometric to a ball in hyperbolic space $\mathbb{H}^3$.
\end{theorem}

For compact free-boundary surfaces (see Section~\ref{jacobi} for the definition), we obtain the following theorem.

\begin{theorem}\label{T1}
Let $M^3$ be a Riemannian manifold with Ricci curvature $\Ric \geq 2$ and convex boundary in the sense that $\II^{\partial M} \geq 0$. Let $\Sigma^2 \subset M^3$ be a compact connected orientable immersed  surface with constant mean curvature and free boundary. If $\lambda_1(J) \geq -2$, then $\Sigma^2$ has genus $0$, a single boundary component, and $|\Sigma^2| \leq 2\pi$.

If equality holds in the area inequality, then $\Sigma^2$ is isometric to $\mathbb{S}^2_+$. Moreover, if the Gaussian curvature of the boundary satisfies $K_{\partial M} \geq 1$, then $M^3$ is isometric to $\mathbb{S}^3_+$.
\end{theorem}

In the case of hypersurfaces with boundary, we also consider the Jacobi--Steklov problem, a Steklov-type problem naturally associated with the Jacobi operator (see Section~\ref{Steklov}). This problem appeared independently in the works of Tran \cite{Tran} and Devyver \cite{Devyver}, where it played an important role in the study of the Morse index of the critical catenoid in the Euclidean ball. Tran also observed that the problem is well defined whenever the Dirichlet problem for $J$ is well posed. This is the case, for instance, under strict stability for the Plateau problem, which is encoded by the positivity of the first Dirichlet eigenvalue of the Jacobi operator, denoted by $\lambda_1^D(J)$.

We denote by $\sigma_1(J)$ the first eigenvalue of the Jacobi--Steklov problem. In Section~\ref{Steklov}, we first establish an estimate in any dimension and characterize the equality case. We then specialize this estimate to dimension two and derive a rigidity theorem for the boundary length.

\begin{theorem}\label{rigidity_n}
Let $(M^{n+1}, g)$ be a compact orientable Riemannian manifold with scalar curvature $S \geq 0$ and boundary mean curvature $H^{\partial M} \geq n$. Let $\Sigma^n \subset M^{n+1}$ be a compact orientable immersed free-boundary hypersurface with constant mean curvature. Assume moreover that $\lambda_1^D(J)>0$. If $\sigma_1(J) \geq -1$, then
\begin{equation}\label{scalarsigma}
2(n-1)|\partial \Sigma|
\leq
\int_{\Sigma} S_\Sigma \, dv + 2\int_{\partial \Sigma} H^{\partial \Sigma, \Sigma} \, da.
\end{equation}
Moreover, equality holds in \eqref{scalarsigma} if and only if $\Sigma$ is totally geodesic, $\Ric(N,N)=0$, $\II^{\partial M}(N,N)=1$, and $S|_{\Sigma}=0$. In particular, $S_\Sigma=0$ and $H^{\partial \Sigma,\Sigma}=n-1$.

Furthermore, assume that $\Sec\ge 0$, $\II^{\partial M}\ge 1$, and that equality holds in \eqref{scalarsigma}. If, in addition, $\partial\Sigma$ is simply connected and $\Sigma$ is properly embedded, then $\Sigma$ is isometric to $\mathbb B^n$ and $M$ is isometric to $\mathbb B^{n+1}$.
\end{theorem}

In dimension two, our boundary-length estimate may be viewed as a Jacobi--Steklov counterpart of the sharp perimeter estimate obtained by Mendes \cite[Theorem 1.4]{Abraao} for index-one free-boundary minimal surfaces. 

\begin{theorem}\label{T5}
Let $M^3$ be a Riemannian manifold with nonempty boundary such that $\Ric \geq 0$ and $\II^{\partial M} \geq 1$. Let $\Sigma^2 \subset M^3$ be a compact orientable immersed free-boundary surface with constant mean curvature. Assume moreover that $\lambda_1^D(J)>0$. If $\sigma_1(J) \geq -1$, then $\Sigma$ has genus zero, exactly one boundary component, and
\[
|\partial\Sigma|\le 2\pi.
\]

Furthermore, if equality holds in the boundary-length estimate, then $\Sigma$ is isometric to the unit disk ${\mathbb D}$. If, in addition, $\Sigma$ is embedded, then $M^3$ is isometric to the Euclidean unit ball ${\mathbb B}^3 \subset \mathbb{R}^3$.
\end{theorem}

In Section~\ref{Yamabe}, we conclude the paper by establishing geometric relations between the eigenvalues $\lambda_1(J)$ and $\sigma_1(J)$ and the Yamabe invariants of manifolds with boundary, inspired by the work of Cai and Galloway \cite{CaiGalloway1}.

%=========================================================
\section{The first eigenvalue of the Jacobi operator: general framework and the dimension-$n$ theorem}\label{jacobi}
%=========================================================

Let $M^{n+1}$ be a connected orientable Riemannian manifold with (possibly empty) smooth boundary, and let $\Sigma^n$ be a compact connected two-sided hypersurface with constant mean curvature immersed in $M^{n+1}$.  We denote by $N$ a global unit normal field along $\Sigma^n$.

If $\partial \Sigma \neq \emptyset$, we assume that $\partial \Sigma$ meets $\partial M$ orthogonally. In such a case, $\Sigma^n$ is a free-boundary hypersurface, and the second variation of area is given by the quadratic form
\begin{equation}\label{eq:indexformula}
\mathcal{I}(u,v)=
-\int_{\Sigma^n}u\,Jv\,dv
+\int_{\partial\Sigma}u\left(\frac{\partial v}{\partial\nu}-\II^{\partial M}(N,N)v\right)da,
\end{equation}
where
\[
J=\Delta+\Ric(N,N)+|A|^2
\]
is the Jacobi operator of $\Sigma^n$ and $A$ is the shape operator of $\Sigma^n$.

We say that $u\in H^1(\Sigma^n)$ is an eigenfunction of the Jacobi operator associated with the eigenvalue $\lambda(J)$ if $u\not\equiv 0$ is a weak solution of
\begin{equation}\label{sff}
\left\{
\begin{array}{ll}
Ju+\lambda(J)u=0,&\text{in }\Sigma^n,\\[1mm]
\dfrac{\partial u}{\partial \nu}=\II^{\partial M}(N,N)u,&\text{on }\partial\Sigma.
\end{array}
\right.
\end{equation}
Equivalently,
\[
\mathcal{I}(u,\phi)=\lambda(J)\langle u,\phi\rangle_{L^2(\Sigma^n)}
\qquad \text{for every }\phi\in H^1(\Sigma^n).
\]

By standard elliptic theory, the eigenvalues of \eqref{sff} form a divergent sequence
\[
\lambda_1(J)<\lambda_2(J)\le \lambda_3(J)\le \cdots \nearrow \infty.
\]
The first eigenvalue admits the variational characterization
\begin{equation}\label{variational}
\lambda_1(J)=
\inf_{u\in H^1(\Sigma^n)\setminus\{0\}}
\frac{\mathcal{I}(u,u)}{\int_{\Sigma^n}u^2\,dv}.
\end{equation}

Although constant mean curvature hypersurfaces arise variationally as critical points of the area functional under a volume constraint, the spectral problem considered in this paper is the unconstrained one associated with the full Jacobi operator. Thus, in the definition of $\lambda_1(J)$ and in the variational characterization \eqref{variational}, no mean-zero condition is imposed on the admissible functions. We refer to Barbosa and B\'erard \cite{BB} for a study of the relation between the spectrum of the volume-preserving problem and that of the usual unconstrained problem; in particular, they show that the two spectra are intertwined. The unconstrained setting is the natural one for our purposes, since the basic estimates below are obtained by testing the Rayleigh quotient with the constant function $u\equiv 1$, which is not admissible in the volume-preserving setting.

When $\partial\Sigma=\emptyset$, all the above statements remain valid after removing the boundary terms.

Tracing the Gauss equation twice, we obtain
\begin{equation}\label{Gauss}
\Ric(N,N)=\frac12\bigl(S-S_\Sigma+H^2-|A|^2\bigr),
\end{equation}
where $H=\tr A$ is the mean curvature of $\Sigma^n$, $S$ is the scalar curvature of $M^{n+1}$, and $S_\Sigma$ is the scalar curvature of $\Sigma^n$.

Taking the constant test function $u\equiv 1$ in \eqref{variational} and using \eqref{Gauss}, we obtain the basic estimate
\begin{equation}\label{eq:basic-general}
\lambda_1(J)\le
-\frac{1}{2|\Sigma^n|}
\int_{\Sigma^n}\bigl(S-S_\Sigma+H^2+|A|^2\bigr)\,dv
-\frac{1}{|\Sigma^n|}\int_{\partial\Sigma}\II^{\partial M}(N,N)\,da.
\end{equation}
This is the starting point for all the estimates in the paper.

The following lemma is well-known among experts, but we include a proof here for completeness.

\begin{lemma}\label{lemma-ab}
Assume there exists $u \in H^{1}(\Sigma^n)\setminus\{0\}$ such that
\[
\mathcal{I}(u,u)=\lambda_1(J)\int_{\Sigma^n}u^2\,dv.
\]
Then $u$ is an eigenfunction associated with $\lambda_1(J)$.
\end{lemma}

\begin{proof}
Given $\phi\in H^1(\Sigma^n)$, consider
\[
f(t):=\mathcal{I}(u+t\phi,u+t\phi)-\lambda_1(J)\int_{\Sigma^n}(u+t\phi)^2\,dv,
\qquad t\in\mathbb R.
\]
By the variational characterization of $\lambda_1(J)$, we have $f(t)\ge 0$ for all $t$. Since $f(0)=0$, the point $t=0$ is a minimum, and therefore
\[
0=f'(0)=2\mathcal{I}(u,\phi)-2\lambda_1(J)\int_{\Sigma^n}u\phi\,dv.
\]
Hence
\[
\mathcal{I}(u,\phi)=\lambda_1(J)\int_{\Sigma^n}u\phi\,dv
\qquad \text{for every }\phi\in H^1(\Sigma^n),
\]
which is exactly the weak formulation of \eqref{sff}.
\end{proof}

We now prove the main theorem of the dimension-$n$ part.

\begin{proof}[Proof of Theorem \ref{T4}]
Since $\Sigma^n$ is closed, \eqref{eq:basic-general} reduces to
\begin{equation}\label{eq:T4-start}
\lambda_1(J)\le
-\frac{1}{2|\Sigma^n|}
\int_{\Sigma^n}\bigl(S-S_\Sigma+H^2+|A|^2\bigr)\,dv.
\end{equation}
Using the lower bound $S\ge n(n+1)$ and the inequalities $H^2\ge 0$ and $|A|^2\ge 0$, together with the assumption $\lambda_1(J)\ge -n$, we obtain
\[
-n\le
-\frac{1}{2|\Sigma^n|}
\int_{\Sigma^n}\bigl(n(n+1)-S_\Sigma\bigr)\,dv.
\]
Rearranging gives
\[
n(n-1)|\Sigma^n|\le \int_{\Sigma^n}S_\Sigma\,dv,
\]
which is \eqref{media}.

Assume now that equality holds in \eqref{media}. Then equality must occur at every step above. In particular,
\[
H\equiv 0,
\qquad
|A|^2\equiv 0,
\qquad
S|_{\Sigma^n}\equiv n(n+1).
\]
Thus $\Sigma^n$ is totally geodesic and the scalar curvature of $M^{n+1}$ along $\Sigma^n$ is constant equal to $n(n+1)$.

Moreover, equality in \eqref{eq:T4-start} means that the constant function $u\equiv 1$ realizes the Rayleigh quotient. By Lemma~\ref{lemma-ab}, the constant function is a first eigenfunction. Hence
\[
J1+\lambda_1(J)\cdot 1=0,
\]
that is,
\[
\Ric(N,N)+|A|^2=-\lambda_1(J).
\]
Since equality in \eqref{media} also forces $\lambda_1(J)=-n$ and $|A|^2=0$, we conclude that
\[
\Ric(N,N)=n
\qquad\text{along }\Sigma^n.
\]

Conversely, if $\Sigma^n$ is totally geodesic, $S|_{\Sigma^n}=n(n+1)$, and $\Ric(N,N)=n$, then $1$ is a positive eigenfunction of $J$ with eigenvalue $-n$, hence $\lambda_1(J)=-n$.

Assume now, in addition, that equality holds and $\Sec\ge 1$. Since $\Sigma^n$ is totally geodesic, the Gauss equation implies that every sectional curvature of $\Sigma^n$ is equal to $1$. Thus $(\Sigma^n,g|_{\Sigma^n})$ has constant sectional curvature $1$. If $n$ is even, then, since $\Sigma^n$ is orientable, Synge's theorem implies that $\Sigma^n$ is simply connected; if $n$ is odd, this is assumed. Hence in either case $\Sigma^n$ is isometric to the round sphere $\mathbb S^n$.

In particular, $\Sigma^n$ is embedded and separates $M^{n+1}$ into two connected components, say $\Omega_1$ and $\Omega_2$; see \cite[Theorem 2]{Lawson}. Applying the Hang--Wang theorem \cite[Theorem 2]{HangWang} to each component, we conclude that both $\Omega_1$ and $\Omega_2$ are isometric to the hemisphere $\mathbb S^{n+1}_+$, and therefore $M^{n+1}$ is isometric to $\mathbb S^{n+1}$.
\end{proof}

%=========================================================
\section{The surface case: Gauss--Bonnet and area rigidity}\label{surface}
%=========================================================

We now specialize the general framework of Section~\ref{jacobi} to dimension two. In this case, Gauss--Bonnet converts the scalar curvature of $\Sigma^2$ into topological information, leading to the surface rigidity statements announced in the Introduction. We first treat the closed case and then the free-boundary case.

Throughout this section, $\Sigma^2$ denotes a compact two-sided CMC surface immersed in a $3$-manifold $M^3$. If $\partial\Sigma\neq\emptyset$, we assume that $\Sigma^2$ has free boundary in $M^3$.

Since $S_\Sigma=2K_\Sigma$, identity \eqref{Gauss} becomes
\begin{equation}\label{eq:Gauss-surface}
\Ric(N,N)=\frac12\bigl(S-2K_\Sigma+H^2-|A|^2\bigr).
\end{equation}
If $\partial\Sigma\neq\emptyset$, then along $\partial\Sigma$ we have
\[
\II^{\partial M}(N,N)=H^{\partial M}-\kappa,
\]
where $H^{\partial M}$ is the mean curvature of $\partial M$ in $M^3$ and $\kappa$ is the geodesic curvature of $\partial\Sigma$ in $\Sigma^2$.

Substituting these identities into \eqref{eq:basic-general} and using Gauss--Bonnet, we obtain the following estimate for $\lambda_1(J)$, together with an infinitesimal rigidity statement in the equality case.

\begin{proposition}\label{Estimate1}
Let $M^3$ be a Riemannian manifold with boundary such that
\[
a:=\inf_M S>-\infty,
\qquad
b:=\inf_{\partial M}H^{\partial M}>-\infty.
\]
Let $\Sigma^2\subset M^3$ be a compact two-sided free-boundary surface with constant mean curvature $H$. Then
\begin{equation}\label{eq:surface-estimate-1}
\lambda_1(J)\le
-\frac12\left(a+\frac32H^2\right)
-b\,\frac{|\partial\Sigma|}{|\Sigma|}
+\frac{2\pi\chi(\Sigma)}{|\Sigma|}.
\end{equation}
Moreover, equality holds if and only if:
\begin{enumerate}
\item $\Sigma^2$ is totally umbilical and the geodesic curvature of $\partial\Sigma$ in $\Sigma^2$ is constant and equal to $b$;
\item $S|_\Sigma=a$, $H^{\partial M}|_{\partial\Sigma}=b$, and $K_\Sigma$ is constant;
\item $\Ric(N,N)=-\lambda_1(J)-\dfrac12H^2$ on $\Sigma$ and $\II^{\partial M}(N,N)=0$ along $\partial\Sigma$.
\end{enumerate}
\end{proposition}

\begin{proof}
By \eqref{eq:basic-general}, \eqref{eq:Gauss-surface}, and Gauss--Bonnet,
\begin{align}
\lambda_1(J)
&\le
-\frac{1}{2|\Sigma|}
\int_{\Sigma}\bigl(S+H^2+|A|^2\bigr)\,dv
-\frac{1}{|\Sigma|}\int_{\partial\Sigma}H^{\partial M}\,da
+\frac{2\pi\chi(\Sigma)}{|\Sigma|}. \label{eq:gb-lambda}
\end{align}
Using $S\ge a$, $H^{\partial M}\ge b$, and $|A|^2\ge \frac12H^2$, we deduce
\[
\lambda_1(J)\le
-\frac12\left(a+\frac32H^2\right)
-b\,\frac{|\partial\Sigma|}{|\Sigma|}
+\frac{2\pi\chi(\Sigma)}{|\Sigma|}.
\]

Assume now that equality holds. Then equality must occur in all the inequalities above. In particular,
\[
|A|^2=\frac12H^2,
\]
so $\Sigma$ is totally umbilical, and also
\[
S|_\Sigma=a,
\qquad
H^{\partial M}|_{\partial\Sigma}=b.
\]
Moreover, equality in the Rayleigh quotient with the test function $u\equiv 1$ implies
\[
\lambda_1(J)=\frac{\mathcal I(1,1)}{|\Sigma|}.
\]
By Lemma~\ref{lemma-ab}, the constant function $1$ is a first eigenfunction. Hence
\[
J1+\lambda_1(J)=0
\quad\text{on }\Sigma,
\qquad
\II^{\partial M}(N,N)=0
\quad\text{along }\partial\Sigma.
\]
Since $J1=\Ric(N,N)+|A|^2$, we get
\[
\Ric(N,N)=-\lambda_1(J)-|A|^2=-\lambda_1(J)-\frac12H^2.
\]
Finally, \eqref{eq:Gauss-surface} shows that $K_\Sigma$ is constant. The converse is immediate.
\end{proof}

In the closed case, the boundary terms disappear, and \eqref{eq:surface-estimate-1} reduces to the estimate proved in \cite{BatistaSantos} for constant density.

We now prove Theorems~\ref{T2}, \ref{T3}, and \ref{T1}.

%---------------------------------------------------------
\subsection{Proof of Theorem \ref{T2}}
%---------------------------------------------------------
\begin{proof}
Since $\Sigma$ is closed, Proposition~\ref{Estimate1} gives
\[
\lambda_1(J)\le -\frac12\left(\frac32H^2+a\right)+\frac{2\pi\chi(\Sigma)}{|\Sigma|}.
\]
Using $a=0$, $H\ge 2$, and $\lambda_1(J)\ge -2$, we obtain
\[
-2\le \lambda_1(J)\le -3+\frac{2\pi\chi(\Sigma)}{|\Sigma|}.
\]
Hence $\chi(\Sigma)>0$, so $\Sigma$ has genus $0$, and therefore
\[
-2\le \lambda_1(J)\le -3+\frac{4\pi}{|\Sigma|},
\]
which yields $|\Sigma|\le 4\pi$.

If equality holds, then necessarily
\[
H=2,
\qquad
\chi(\Sigma)=2,
\qquad
|\Sigma|=4\pi.
\]
By the equality case in Proposition~\ref{Estimate1}, $\Sigma$ is totally umbilical and has constant Gaussian curvature $K_\Sigma=1$. Since it has genus $0$, it is isometric to the unit sphere $\mathbb S^2$. 
If $\Sigma$ bounds a domain in $M^3$, the rigidity statement follows from the rigidity theorem for the Euclidean ball proved by Shi--Tam \cite{ShiTam2} and Miao \cite{Miao}; see \cite[Theorem 1.1]{HangWang06}.
\end{proof}

%---------------------------------------------------------
\subsection{Proof of Theorem \ref{T3}}
%---------------------------------------------------------
\begin{proof}
Again, since $\Sigma$ is closed, Proposition~\ref{Estimate1} yields
\[
\lambda_1(J)\le -\frac12\left(\frac32H^2-6\right)+\frac{2\pi\chi(\Sigma)}{|\Sigma|}.
\]
Using $H\ge 2\sqrt2$ and $\lambda_1(J)\ge -2$, we obtain
\[
-2\le \lambda_1(J)\le -3+\frac{2\pi\chi(\Sigma)}{|\Sigma|}.
\]
Thus $\chi(\Sigma)>0$, so $\Sigma$ has genus $0$, and consequently
\[
-2\le \lambda_1(J)\le -3+\frac{4\pi}{|\Sigma|},
\]
which implies $|\Sigma|\le 4\pi$.

If equality holds, then
\[
H=2\sqrt2,
\qquad
\chi(\Sigma)=2,
\qquad
|\Sigma|=4\pi.
\]
By the equality case in Proposition~\ref{Estimate1}, $\Sigma$ is totally umbilical and has constant Gaussian curvature $K_\Sigma=1$, hence it is isometric to $\mathbb S^2$. The rigidity statement for the bounded domain then follows from \cite[Theorem 3.8]{ShiTam}.
\end{proof}

%---------------------------------------------------------
\subsection{Proof of Theorem \ref{T1}}
%---------------------------------------------------------
\begin{proof}
Since $\Ric\ge 2$, the scalar curvature of $M^3$ satisfies $S\ge 6$. Moreover, the convexity assumption $\II^{\partial M}\ge 0$ implies $H^{\partial M}\ge 0$. Using only the weaker bound $|A|^2\ge 0$ in \eqref{eq:gb-lambda}, we obtain
\[
\lambda_1(J)\le -\frac12(H^2+6)+\frac{2\pi\chi(\Sigma)}{|\Sigma|}.
\]
Writing $\chi(\Sigma)=2-2g-r$, where $g$ is the genus and $r$ is the number of boundary components, and using $\lambda_1(J)\ge -2$, we obtain
\[
-2\le \lambda_1(J)\le -3+\frac{2\pi(2-2g-r)}{|\Sigma|}
\le -3+\frac{2\pi}{|\Sigma|}.
\]
Hence $2-2g-r>0$, which forces $g=0$ and $r=1$, and also $|\Sigma|\le 2\pi$.

Assume now that equality holds. Then equality holds throughout the above chain, so in particular
\[
H=0,
\qquad
\chi(\Sigma)=1,
\qquad
|\Sigma|=2\pi.
\]
Moreover, equality in the estimate implies that $\Sigma$ is totally geodesic, $S|_\Sigma=6$, $H^{\partial M}|_{\partial\Sigma}=0$, and $K_\Sigma$ is constant. Since
\[
K_\Sigma=\frac12S|_\Sigma-\Ric(N,N)=1,
\]
the surface $\Sigma$ has constant Gaussian curvature $1$ and totally geodesic boundary. It then follows from \cite[Theorem 4]{HangWang} that $\Sigma$ is isometric to $\mathbb S^2_+$.

Finally, assume in addition that $K_{\partial M}\ge 1$. Since $\partial\Sigma$ is a geodesic circle of length $2\pi$ in $\partial M$, Toponogov's theorem \cite{Toponogov} (see also \cite[Corollary 1]{HangWang}) implies that $\partial M$ is isometric to $\mathbb S^2$. The Hang--Wang theorem \cite[Theorem 2]{HangWang} then shows that $M^3$ is isometric to the hemisphere $\mathbb S^3_+$.
\end{proof}

%=========================================================
\section{The Jacobi--Steklov problem}\label{Steklov}
%=========================================================

In this section we introduce a Steklov-type spectral problem naturally associated with the Jacobi operator of a free-boundary CMC hypersurface. This problem has already appeared in the literature; see, for example, \cite{AntoniaCavalcanteSouza,Devyver,Tran}. Following the same organizational principle as in the previous sections, we first work in general dimension and prove the higher-dimensional result stated in the Introduction. We then specialize to dimension two and derive the corresponding boundary-length estimate.

%---------------------------------------------------------
\subsection{The boundary value problem and the Dirichlet-to-Neumann operator}\label{subsec:DtN-J}
%---------------------------------------------------------

We consider the Steklov-type eigenvalue problem
\begin{equation}\label{eq:Jacobi-Steklov}
\left\{
\begin{array}{rcll}
Ju&=&0&\text{in }\Sigma^n,\\[1mm]
B_0u&=&\sigma\,u&\text{on }\partial\Sigma,
\end{array}
\right.
\end{equation}
where $J$ is the Jacobi operator of $\Sigma^n\subset M^{n+1}$ and
\[
B_0=\partial_\nu-\II^{\partial M}(N,N)
\]
is the boundary operator arising in the second variation of area.

We say that $u\in H^1(\Sigma^n)$ is a \emph{Jacobi--Steklov eigenfunction} associated with $\sigma\in\mathbb R$ if $u\not\equiv 0$ is a weak solution of \eqref{eq:Jacobi-Steklov}. We denote the corresponding eigenvalues by
\[
\sigma_1(J)\le \sigma_2(J)\le \cdots .
\]

Unlike the classical Steklov problem for $\Delta$, the Dirichlet problem for $J$ need not be solvable for arbitrary boundary data. To define the Dirichlet-to-Neumann map without domain restrictions, we assume that the first Dirichlet eigenvalue is positive.

\begin{definition}\label{def:Dirichlet-ground}
Let $\lambda_1^D(J)$ denote the first eigenvalue of the Dirichlet problem
\[
\left\{
\begin{aligned}
Ju+\lambda u&=0&&\text{in }\Sigma^n,\\
u&=0&&\text{on }\partial\Sigma.
\end{aligned}
\right.
\]
\end{definition}

\begin{assumption}\label{ass:coercive}
Throughout this section, we assume that
\begin{equation}\label{eq:Dirichlet-positive}
\lambda_1^D(J)>0.
\end{equation}
\end{assumption}

Under \eqref{eq:Dirichlet-positive}, the Dirichlet quadratic form
\[
Q^D[u]
:=\int_{\Sigma^n}|\nabla u|^2\,dv
-\int_{\Sigma^n}\big(\Ric(N,N)+|A|^2\big)u^2\,dv,
\qquad u\in H_0^1(\Sigma^n),
\]
is coercive. In particular, the Dirichlet problem for $J$ is well posed for every boundary trace. Geometrically, \eqref{eq:Dirichlet-positive} means that $\Sigma^n$ is strictly stable for the Plateau problem.

The following well-posedness statement is standard; for completeness, we refer the reader to \cite{AntoniaCavalcanteSouza}.

\begin{proposition}\label{prop:wellposed}
Assume \eqref{eq:Dirichlet-positive}. Given $h\in H^{1/2}(\partial\Sigma)$, set
\[
\mathcal A_h:=\{u\in H^1(\Sigma^n): u|_{\partial\Sigma}=h\}.
\]
Then there exists a unique $\widehat h\in \mathcal A_h$ minimizing $Q^D$ on $\mathcal A_h$.
Moreover, $\widehat h$ is the unique weak solution of $Ju=0$ in $\Sigma^n$ with trace $h$.
\end{proposition}

We define the Dirichlet-to-Neumann operator for $J$ by
\[
\Lambda_J(h):=\left.\frac{\partial \widehat h}{\partial \nu}\right|_{\partial\Sigma},
\]
where $\widehat h$ is the $J$-harmonic extension given by Proposition~\ref{prop:wellposed}. Then \eqref{eq:Jacobi-Steklov} is equivalent to the boundary spectral problem
\[
\big(\Lambda_J-\II^{\partial M}(N,N)\big)h=\sigma\,h
\qquad\text{on }\partial\Sigma,
\qquad h=u|_{\partial\Sigma}.
\]
In particular, under \eqref{eq:Dirichlet-positive}, the Jacobi--Steklov spectrum is discrete and real.

Introduce the quadratic form
\begin{equation}\label{eq:Q-JS}
Q[u]
:=
\int_{\Sigma^n} |\nabla u|^2\,dv
-\int_{\Sigma^n}\big(\Ric(N,N)+|A|^2\big)u^2\,dv
-\int_{\partial\Sigma}\II^{\partial M}(N,N)\,u^2\,da.
\end{equation}
If $Ju=0$ in $\Sigma^n$, then integration by parts yields
\begin{equation}\label{eq:Q-boundary-identity}
Q[u]
=\int_{\partial\Sigma}u\Big(\frac{\partial u}{\partial \nu}-\II^{\partial M}(N,N)\,u\Big)da
=\int_{\partial\Sigma}u\,B_0u \,da.
\end{equation}
Hence, for a Jacobi--Steklov eigenfunction, $Q[u]=\sigma\int_{\partial\Sigma}u^2\,da$.

By minimizing over boundary traces and using Proposition~\ref{prop:wellposed}, we obtain
\begin{equation}\label{eq:sigma1-var-correct}
\sigma_1(J)
=
\inf_{h\in H^{1/2}(\partial\Sigma)\setminus\{0\}}
\frac{Q[\widehat h]}{\int_{\partial\Sigma}h^2\,da},
\end{equation}
where $\widehat h$ is the $J$-harmonic extension of $h$.

%---------------------------------------------------------
\subsection{A two-dimensional estimate for \texorpdfstring{$\sigma_1(J)$}{sigma1(J)}}\label{subsec:sigma1-surface-estimate}
%---------------------------------------------------------

Before proving the general rigidity statement, we record the following estimate in dimension two, which is also of independent interest.

\begin{theorem}\label{Estimate3}
Let $M^3$ be a Riemannian manifold with boundary such that
\[
a:=\inf_M S
\qquad\text{and}\qquad
b:=\inf_{\partial M} H^{\partial M}
\]
are finite. Let $\Sigma^2\subset M^3$ be a compact two-sided free-boundary surface with constant mean curvature $H$. Assume moreover that $\lambda_1^D(J)>0$. Then
\begin{equation}\label{sigma1}
\sigma_1(J)
\le
-\left(\frac{a}{2}+\frac{3}{4}H^2\right)\frac{|\Sigma|}{|\partial\Sigma|}
+\frac{2\pi\chi(\Sigma)}{|\partial\Sigma|}
-b.
\end{equation}
Moreover, equality holds in \eqref{sigma1} if and only if the following properties are satisfied:
\begin{enumerate}
\item $\Sigma$ is totally umbilical and has constant Gaussian curvature;
\item $S|_{\Sigma}=a$ and $H^{\partial M}|_{\partial\Sigma}=b$;
\item $\Ric(N,N)=-\dfrac{1}{2}H^2$ along $\Sigma$ and
\[
\II^{\partial M}(N,N)=-\sigma_1(J)
\qquad\text{along }\partial\Sigma.
\]
\end{enumerate}
\end{theorem}

\begin{proof}
Let $h\equiv 1$ on $\partial\Sigma$, and let $\widehat 1$ be its $J$-harmonic extension. By \eqref{eq:sigma1-var-correct},
\[
\sigma_1(J)\le \frac{Q[\widehat 1]}{|\partial\Sigma|}.
\]
Since $\widehat 1$ minimizes $Q^D$ among all functions with boundary value $1$, we have
\[
Q[\widehat 1]\le Q[1].
\]
Therefore,
\begin{equation}\label{eq:sigma1Q1-Estimate3}
\sigma_1(J)\le \frac{Q[1]}{|\partial\Sigma|}.
\end{equation}

Now,
\begin{align*}
Q[1]
&=
-\int_\Sigma \big(\Ric(N,N)+|A|^2\big)\,dv
-\int_{\partial\Sigma}\II^{\partial M}(N,N)\,da.
\end{align*}
Using the Gauss equation
\[
\Ric(N,N)+|A|^2
=
\frac12\big(S-2K_\Sigma+H^2+|A|^2\big),
\]
together with the free-boundary identity
\[
\II^{\partial M}(N,N)=H^{\partial M}-\kappa,
\]
we obtain
\begin{align*}
Q[1]
&=
-\frac12\int_\Sigma \big(S-2K_\Sigma+H^2+|A|^2\big)\,dv
-\int_{\partial\Sigma}\big(H^{\partial M}-\kappa\big)\,da \\
&=
-\frac12\int_\Sigma \big(S+H^2+|A|^2\big)\,dv
+\int_\Sigma K_\Sigma\,dv
-\int_{\partial\Sigma}H^{\partial M}\,da
+\int_{\partial\Sigma}\kappa\,da.
\end{align*}
Since $|A|^2\ge \dfrac{H^2}{2}$, $S\ge a$, and $H^{\partial M}\ge b$, Gauss--Bonnet yields
\begin{align*}
Q[1]
&\le
-\frac12\left(a+\frac32 H^2\right)|\Sigma|
+2\pi\chi(\Sigma)
-b|\partial\Sigma|.
\end{align*}
Combining this with \eqref{eq:sigma1Q1-Estimate3}, we obtain \eqref{sigma1}.

Assume now that equality holds in \eqref{sigma1}. Then equality must occur at every step above. In particular:
\begin{itemize}
\item $Q[\widehat1]=Q[1]$, hence $\widehat1\equiv 1$ by uniqueness of the $J$-harmonic extension with boundary value $1$;
\item $|A|^2=\dfrac{H^2}{2}$ on $\Sigma$, so $\Sigma$ is totally umbilical;
\item $S|_\Sigma=a$ and $H^{\partial M}|_{\partial\Sigma}=b$;
\item $K_\Sigma$ is constant, because the Gauss equation then gives
\[
K_\Sigma
=
\frac12\big(S+H^2-|A|^2\big)-\Ric(N,N)
=
\frac12\left(a+\frac12 H^2\right)-\Ric(N,N),
\]
and, as we now show, $\Ric(N,N)$ is constant as well.
\end{itemize}

Since $\widehat1\equiv 1$ is a Jacobi--Steklov eigenfunction associated with $\sigma_1(J)$, the boundary value problem \eqref{eq:Jacobi-Steklov} gives
\[
J1=0 \quad\text{in }\Sigma,
\qquad
B_0 1=\sigma_1(J)\quad\text{on }\partial\Sigma.
\]
Thus
\[
\Ric(N,N)+|A|^2=0
\quad\text{on }\Sigma
\]
and
\[
-\II^{\partial M}(N,N)=\sigma_1(J)
\quad\text{on }\partial\Sigma.
\]
Since $|A|^2=\dfrac{H^2}{2}$, we conclude that
\[
\Ric(N,N)=-\frac12 H^2
\quad\text{on }\Sigma.
\]
Substituting this into the Gauss equation, we obtain
\[
K_\Sigma=\frac12\left(a+\frac32 H^2\right),
\]
which is constant. This proves (1)--(3). The converse is immediate, since under (1)--(3) the function $1$ is a Jacobi--Steklov eigenfunction and all inequalities above become equalities.
\end{proof}

%---------------------------------------------------------
\subsection{The dimension-$n$ estimate and rigidity}\label{subsec:sigma1-general}
%---------------------------------------------------------

We now prove the higher-dimensional estimate announced in the Introduction.

\begin{proof}[Proof of Theorem \ref{rigidity_n}]
Let $h\equiv 1$ on $\partial\Sigma$ and let $\widehat 1$ be its $J$-harmonic extension. By \eqref{eq:sigma1-var-correct},
\[
\sigma_1(J)\le \frac{Q[\widehat 1]}{|\partial\Sigma|}.
\]
Since $\widehat 1$ minimizes the Dirichlet energy among all extensions with trace $1$, we have
\[
Q[\widehat 1]\le Q[1].
\]
Therefore,
\begin{equation}\label{eq:sigma-general-byQ1}
\sigma_1(J)\le \frac{Q[1]}{|\partial\Sigma|}.
\end{equation}

Using \eqref{Gauss} and the free-boundary identity
\[
\II^{\partial M}(N,N)=H^{\partial M}-H^{\partial\Sigma,\Sigma},
\]
we compute
\begin{align*}
Q[1]
&=
-\int_{\Sigma}\bigl(\Ric(N,N)+|A|^2\bigr)\,dv
-\int_{\partial\Sigma}\II^{\partial M}(N,N)\,da \\
&=
-\frac12\int_{\Sigma}\bigl(S-S_\Sigma+H^2+|A|^2\bigr)\,dv
-\int_{\partial\Sigma}\bigl(H^{\partial M}-H^{\partial\Sigma,\Sigma}\bigr)\,da \\
&=
\frac12\int_{\Sigma}S_\Sigma\,dv
+\int_{\partial\Sigma}H^{\partial\Sigma,\Sigma}\,da
-\frac12\int_{\Sigma}\bigl(S+H^2+|A|^2\bigr)\,dv
-\int_{\partial\Sigma}H^{\partial M}\,da.
\end{align*}
Since $S\ge 0$, $H^{\partial M}\ge n$, $H^2\ge 0$, and $|A|^2\ge 0$, we obtain
\[
Q[1]
\le
\frac12\int_{\Sigma}S_\Sigma\,dv
+\int_{\partial\Sigma}H^{\partial\Sigma,\Sigma}\,da
-n|\partial\Sigma|.
\]
Combining this with \eqref{eq:sigma-general-byQ1} and the assumption $\sigma_1(J)\ge -1$, we get
\[
-|\partial\Sigma|
\le
\frac{1}{2}\int_{\Sigma}S_\Sigma\,dv
+\int_{\partial\Sigma}H^{\partial\Sigma,\Sigma}\,da
-n|\partial\Sigma|,
\]
which is equivalent to
\[
2(n-1)|\partial\Sigma|
\le
\int_{\Sigma}S_\Sigma\,dv+2\int_{\partial\Sigma}H^{\partial\Sigma,\Sigma}\,da.
\]

Assume now that equality holds. Then equality must occur in all the inequalities above. In particular,
\[
H\equiv 0,\qquad |A|^2\equiv 0,\qquad S|_{\Sigma}\equiv 0,\qquad H^{\partial M}|_{\partial\Sigma}\equiv n.
\]
Thus $\Sigma$ is totally geodesic. Moreover, equality in \eqref{eq:sigma-general-byQ1} implies that $Q[\widehat1]=Q[1]$, hence $\widehat1\equiv 1$ by uniqueness of the $J$-harmonic extension with boundary value $1$. Therefore
\[
J1=0\quad\text{on }\Sigma,
\qquad
B_01=\sigma_1(J)\quad\text{on }\partial\Sigma.
\]
Since equality also forces $\sigma_1(J)=-1$, we obtain
\[
\Ric(N,N)+|A|^2=0,
\qquad
\II^{\partial M}(N,N)=1.
\]
Because $|A|^2=0$, we get $\Ric(N,N)=0$. Moreover, because $\Sigma$ is totally geodesic and $S|_\Sigma=0$, the Gauss equation yields
\[
S_\Sigma = S - 2\Ric(N,N) + H^2 - |A|^2 = 0.
\]

Finally, using again the free-boundary condition,
\[
H^{\partial\Sigma,\Sigma}
=
H^{\partial M}-\II^{\partial M}(N,N)
=
n-1.
\]
This proves the equality statement.

Assume now, in addition, that $\Sec\ge 0$, $\II^{\partial M}\ge 1$, and that equality holds in \eqref{scalarsigma}. Since $\Sigma$ is totally geodesic, the Gauss equation implies that $\Sec_\Sigma\ge 0$. As $S_\Sigma=0$, it follows that every sectional curvature of $\Sigma$ vanishes. Hence $\Sigma$ is flat.

Moreover, along $\partial\Sigma$, the inequality $\II^{\partial M}\ge 1$ implies that each principal curvature of $\partial\Sigma\subset\Sigma$ is at least $1$. Since
\[
H^{\partial\Sigma,\Sigma}=n-1,
\]
all these principal curvatures must be equal to $1$. Therefore,
$\II^{\partial\Sigma,\Sigma}=g_{\partial\Sigma}.$
Applying the Gauss equation to the hypersurface $\partial\Sigma\subset\Sigma$, we conclude that every sectional curvature of $\partial\Sigma$ is equal to $1$. Since $\partial\Sigma$ is simply connected, it follows that $\partial\Sigma$ is isometric to $\mathbb S^{n-1}$.

We may now apply Xia's theorem \cite[Theorem 1]{Xia} to conclude that $\Sigma$ is isometric to the Euclidean unit ball $\mathbb B^n$.

Assume now that $\Sigma$ is properly embedded. Since equality implies that $\Sigma$ is minimal we can use a result of Fraser and Li \cite[Corollary 2.10]{FraserLi} to show that $\Sigma$ separates $M$ into two connected components,
\[
M\setminus\Sigma=M_+\cup M_-.
\]
For each component $M_\pm$, the boundary is made up of $\Sigma$ and a smooth piece $\Gamma_\pm\subset\partial M$, meeting orthogonally along $\partial\Sigma$. Since $\Ric\ge 0$, the mean curvature $H_{\Gamma_\pm}\ge n$, and $\Sigma$ is totally geodesic and isometric to $\mathbb B^n$, all the hypotheses of the Mazet--Mendes half-ball rigidity theorem \cite[Theorem 5.1]{MM23} are satisfied. Therefore each $M_\pm$ is isometric to the Euclidean unit half-ball $\mathbb B^{n+1}_+$, and consequently $M$ is isometric to the Euclidean unit ball $\mathbb B^{n+1}$.

\end{proof}

%---------------------------------------------------------
\subsection{The surface consequence}\label{subsec:sigma1-surface}
%---------------------------------------------------------

The two-dimensional estimate is the specialization of Theorem~\ref{rigidity_n} to $n=2$, together with Gauss--Bonnet.

\begin{proof}[Proof of Theorem \ref{T5}]
Applying Theorem~\ref{rigidity_n} with $n=2$, we obtain
\[
2|\partial\Sigma|
\le
\int_{\Sigma}S_\Sigma\,dv+2\int_{\partial\Sigma}H^{\partial\Sigma,\Sigma}\,da.
\]
Since $S_\Sigma=2K_\Sigma$, Gauss--Bonnet gives
\[
\int_{\Sigma}S_\Sigma\,dv+2\int_{\partial\Sigma}H^{\partial\Sigma,\Sigma}\,da
=
2\int_{\Sigma}K_\Sigma\,dv+2\int_{\partial\Sigma}\kappa\,da
=
4\pi\chi(\Sigma),
\]
hence
\[
|\partial\Sigma|
\le
2\pi\chi(\Sigma)
=
2\pi(2-2g-r),
\]
where $g$ is the genus and $r$ is the number of boundary components. Since the left-hand side is positive, it follows that $g=0$ and $r=1$. Consequently,
\[
|\partial\Sigma|\le 2\pi.
\]

Assume now that equality holds in the boundary-length inequality. Then equality holds in the estimate above, hence equality holds in Theorem~\ref{rigidity_n}. Therefore,
\[
\Sigma \text{ is totally geodesic},\qquad
K_\Sigma=0,\qquad
\kappa=H^{\partial\Sigma,\Sigma}=1.
\]
Since $\Sigma$ is topologically a disk, it follows that $\Sigma$ is isometric to the Euclidean unit disk ${\mathbb D}$. 
The ambient rigidity follows as in the proof of Theorem~\ref{rigidity_n}.
\end{proof}

%=========================================================
\section{Estimates involving the Yamabe invariants} \label{Yamabe}
%=========================================================

It is noteworthy that, for hypersurfaces with boundary, the first eigenvalue of the Jacobi operator, $\lambda_1(J)$, and the first eigenvalue of the Jacobi--Steklov problem, $\sigma_1(J)$, are related to the Yamabe invariants. These geometric invariants were introduced by Escobar in \cite{E92,Escobar92} and play a central role in the Yamabe problem with boundary.

Let $(\Sigma^n, g)$ be a compact Riemannian manifold with boundary $\partial \Sigma$ and dimension $n \geq 3$. Following Escobar \cite{Escobar92}, we define  the functional
\begin{eqnarray*}
 Q_g (\varphi) =
 \dfrac{\int_{\Sigma^n} \left(a_n|\nabla \varphi|^2 + S_\Sigma \varphi^2 \right)dv + 2\int_{\partial \Sigma}  H^{\partial\Sigma, \Sigma} \varphi^2 da }
 {\left(\int_{\Sigma^n} |\varphi|^{\frac{2n}{n-2}}dv\right)^{\frac{n-2}{n}}},
\end{eqnarray*}
where $\varphi$ is a nonzero smooth function on $\Sigma^n$, and $a_n = \frac{4(n-1)}{n-2}$.

The first \emph{Yamabe constant for manifolds with boundary} is defined by
\begin{eqnarray*}
 Q_g (\Sigma) = 
 \inf \left\{Q_g (\varphi): \varphi \in C^{\infty}(\Sigma)\setminus \{0\}\right\}.
\end{eqnarray*}
It is not difficult to see that $Q_g (\Sigma)$ is invariant under conformal changes, and therefore, depends only on the conformal class $[g]$ of $g$ (cf. \cite{Escobar92}). Moreover,
\[
Q_g(\Sigma) \leq Q_g(\mathbb{S}^n_+).
\]

The first \emph{Yamabe invariant of} $\Sigma^n$ is
\begin{eqnarray*}
 \sigma ( \Sigma) = \sup_{ [g] \in \mathcal{C}(\Sigma) }  Q_g (\Sigma),
\end{eqnarray*}
where $\mathcal{C} (\Sigma)$ is the space of conformal classes on $\Sigma^n$.

Armed with this invariant, we obtain the following boundary analogue of \cite[Theorem 6]{CaiGalloway1}.

\begin{theorem}\label{202} Let $(M^{n+1}, g)$ be an orientable Riemannian manifold with boundary, $n\geq 3$, with scalar curvature $S\geq a$, and mean-convex boundary, $H^{\partial M}\geq 0$. Let $\Sigma^{n} \subset M^{n+1}$ be a compact orientable hypersurface with constant mean curvature $H$ and free boundary. If $\sigma(\Sigma)<0$, then $2\lambda_1(J)+a+H^2<0$ and \begin{equation}\label{201} \left| \Sigma \right| \geq \left(\frac{\sigma(\Sigma)}{2\lambda_1(J)+a+H^2}\right)^{\frac{n}{2}}. \end{equation} Moreover, equality in \eqref{201} holds if and only if $\sigma(\Sigma)=Q_g(1)$, $\Sigma^n$ is totally geodesic in $M^{n+1}$, $S\big|_{\Sigma^n}=a$, $H^{\partial M}\big|_{\partial\Sigma}=0$, $\mathrm{II}^{\partial M}(N,N)\big|_{\partial\Sigma}=0$, and $\Ric(N,N)=-\lambda_1(J)$. In particular, $S_\Sigma=a+2\lambda_1(J)<0$, $H^{\partial\Sigma,\Sigma}=0$, and $(\Sigma^n, g\big|_{\Sigma^n})$ is Einstein with totally geodesic boundary. 
\end{theorem}

\begin{proof}
Using the variational characterization \eqref{variational} of $\lambda_1(J)$, the fact that
\[
a_n=\frac{4(n-1)}{n-2}>2,
\]
and our curvature assumptions, we obtain
\begin{align*}
2\lambda_1(J)\int_{\Sigma} \varphi^2 \, dv
&\leq
2\int_{\Sigma} \left(|\nabla \varphi|^2-\bigl(\Ric(N,N)+|A|^2\bigr)\varphi^2\right) dv
-2\int_{\partial \Sigma} \II^{\partial M}(N,N)\varphi^2\,da \\
&\leq
\int_{\Sigma}\left(a_n|\nabla\varphi|^2-\bigl(S-S_\Sigma+H^2+|A|^2\bigr)\varphi^2\right)dv \\
&\qquad
-2\int_{\partial\Sigma}\bigl(H^{\partial M}-H^{\partial\Sigma,\Sigma}\bigr)\varphi^2\,da \\
&\leq
\int_{\Sigma}\left(a_n|\nabla\varphi|^2+S_\Sigma\varphi^2\right)dv
-\int_{\Sigma}(a+H^2)\varphi^2\,dv
+2\int_{\partial\Sigma}H^{\partial\Sigma,\Sigma}\varphi^2\,da .
\end{align*}
Therefore,
\begin{align*}
(2\lambda_1(J)+a+H^2)\int_{\Sigma}\varphi^2\,dv
&\leq
\int_{\Sigma}\left(a_n|\nabla\varphi|^2+S_\Sigma\varphi^2\right)dv
+2\int_{\partial\Sigma}H^{\partial\Sigma,\Sigma}\varphi^2\,da \\
&=
Q_{g|_\Sigma}(\varphi)\left(\int_{\Sigma}\varphi^{\frac{2n}{n-2}}\,dv\right)^{\frac{n-2}{n}} .
\end{align*}

Since $\sigma(\Sigma)<0$, we can choose a sequence of smooth functions $\{\varphi_\ell\}$ such that
\[
Q_{g|_\Sigma}(\varphi_\ell)<0
\qquad\text{and}\qquad
Q_{g|_\Sigma}(\varphi_\ell)\to \sigma(\Sigma).
\]
In particular, this implies
\[
2\lambda_1(J)+a+H^2<0.
\]
By Hölder's inequality,
\[
\int_{\Sigma}\varphi_\ell^2\,dv
\le |\Sigma|^{\frac{2}{n}}
\left(\int_{\Sigma}\varphi_\ell^{\frac{2n}{n-2}}\,dv\right)^{\frac{n-2}{n}},
\]
and since $Q_{g|_\Sigma}(\varphi_\ell)<0$, we get
\[
(2\lambda_1(J)+a+H^2)|\Sigma|^{\frac{2}{n}}
\le Q_{g|_\Sigma}(\varphi_\ell).
\]
Letting $\ell\to\infty$, we obtain \eqref{201}.

Assume now that equality holds in \eqref{201}. Then equality must occur in all the inequalities above. Hence
\[
|A|^2=0,
\qquad
S|_{\Sigma}=a,
\qquad
H^{\partial M}|_{\partial\Sigma}=0,
\qquad
\II^{\partial M}(N,N)|_{\partial\Sigma}=0,
\qquad
\Ric(N,N)=-\lambda_1(J),
\]
and also
\[
Q_{g|_\Sigma}(1)=\sigma(\Sigma).
\]
Since $|A|^2=0$, the hypersurface $\Sigma$ is totally geodesic in $M$. Moreover,
\[
S_\Sigma=S|_{\Sigma}-2\Ric(N,N)=a+2\lambda_1(J)<0.
\]
Using the free-boundary relation,
\[
H^{\partial\Sigma,\Sigma}
=
H^{\partial M}|_{\partial\Sigma}-\II^{\partial M}(N,N)|_{\partial\Sigma}=0.
\]

Finally, $Q_{g|_\Sigma}(1)=\sigma(\Sigma)$ shows that the induced metric $g|_\Sigma$ realizes the Yamabe invariant $\sigma(\Sigma)$. Since the normalized Yamabe quotient is invariant under constant homotheties, after rescaling $g|_\Sigma$ to have unit volume we may apply \cite[Proposition 5.3]{CruzSantos} and conclude that $(\Sigma^n,g|_\Sigma)$ is Einstein with totally geodesic boundary.
\end{proof}

Escobar also introduced a second invariant (see \cite{E92}), which we now describe. Consider the functional
\begin{equation}\label{yx}
 Y(\varphi) =
 \dfrac{\int_{\Sigma^n} \left(a_n|\nabla \varphi|^2 + S_\Sigma \varphi^2 \right)dv + 2\int_{\partial \Sigma}  H^{\partial\Sigma,\Sigma} \varphi^2 \, da }
 {\left(\int_{\partial\Sigma} |\varphi|^{\frac{2(n-1)}{n-2}}\,da\right)^{\frac{n-2}{n-1}}},
\end{equation}
where $\varphi \in C^\infty(\Sigma^n)$ and $\varphi|_{\partial\Sigma}\not\equiv 0$.

The second \emph{Yamabe constant for manifolds with boundary} is defined by
\[
Y(\Sigma,\partial\Sigma)
=
\inf \left\{
Y(\varphi) : \varphi\in C^\infty(\Sigma^n),\ \varphi|_{\partial\Sigma}\not\equiv 0
\right\}.
\]
Again, it depends only on the conformal class $[g]$; see \cite{E92}. Moreover,
\[
Y(\Sigma,\partial \Sigma) \leq Y(\mathbb{B}^n,\partial \mathbb{B}^n).
\]

The second \emph{Yamabe invariant of} $\Sigma^n$ is
\begin{eqnarray*}
 \tau ( \Sigma,\partial\Sigma) = \sup_{ [g] \in \mathcal{C}(\Sigma) }  Y(\Sigma, \partial\Sigma).
\end{eqnarray*}

Using this invariant we obtain the following estimate.

\begin{theorem}
Let $(M^{n+1}, g)$ be an orientable Riemannian manifold with boundary, $n\geq 3$, with scalar curvature $S\geq 0$ and $H^{\partial M}\geq b$. Let $ \Sigma^{n} \subset M^{n+1}$ be a compact orientable hypersurface with constant mean curvature $H$ and free boundary.
Assume \eqref{eq:Dirichlet-positive} and that $\tau(\Sigma,\partial\Sigma) < 0$. Then $\sigma_1(J) +b<0$ and
\begin{equation}\label{305}
\left| \partial \Sigma \right|\geq \left(\frac{\tau(\Sigma, \partial \Sigma)}{2(\sigma_1(J) +b)}\right)^{n-1}.
\end{equation}
Moreover, equality in \eqref{305} holds if and only if $\tau(\Sigma,\partial\Sigma)=Y(1)$, $\Sigma^n$ is totally geodesic in $M^{n+1}$, $S\big|_{\Sigma^n}=0$, $H^{\partial M}\big|_{\partial\Sigma}=b$, $\II^{\partial M}(N,N)\big|_{\partial\Sigma}=-\sigma_1(J)$, and $\Ric(N,N)=0$.
\end{theorem}

\begin{proof}
Fix $\psi\in C^\infty(\Sigma^n)$ with $\psi\not\equiv 0$ on $\partial\Sigma$ and set $h:=\psi|_{\partial\Sigma}$.
Let $\widehat h$ be the $J$-harmonic extension given by Proposition~\ref{prop:wellposed}. By \eqref{eq:sigma1-var-correct},
\[
\sigma_1(J)\le \frac{Q[\widehat h]}{\int_{\partial\Sigma}h^2}.
\]
Since $\widehat h$ minimizes the Dirichlet energy among all extensions with trace $h$, we have
\[
Q[\widehat h]\le Q[\psi].
\]
Hence
\begin{equation}\label{eq:sigma-psi}
\sigma_1(J)\int_{\partial\Sigma}\psi^2\,da
\le Q[\psi]
=
\int_{\Sigma^n}\left(|\nabla\psi|^2-(\Ric(N,N)+|A|^2)\psi^2\right)dv
-\int_{\partial\Sigma}\II^{\partial M}(N,N)\psi^2\,da.
\end{equation}

Using \eqref{Gauss}, the inequality $a_n>2$, and the assumptions $S\ge 0$ and $H^{\partial M}\ge b$, we estimate the right-hand side exactly as in the proof of Theorem~\ref{202}, now with the boundary terms included, and obtain
\[
2(\sigma_1(J)+b)\int_{\partial\Sigma}\psi^2\,da
\le
\int_{\Sigma^n}\left(a_n|\nabla\psi|^2+S_\Sigma\psi^2\right)dv
+2\int_{\partial\Sigma}H^{\partial\Sigma,\Sigma}\psi^2\,da.
\]
Therefore,
\[
2(\sigma_1(J)+b)\int_{\partial\Sigma}\psi^2\,da
\le
Y(\psi)\left(\int_{\partial\Sigma}|\psi|^{\frac{2(n-1)}{n-2}}\,da\right)^{\frac{n-2}{n-1}}.
\]
If $Y(\psi)<0$, then Hölder's inequality gives
\[
\int_{\partial\Sigma}\psi^2\,da
\le |\partial\Sigma|^{\frac{1}{n-1}}
\left(\int_{\partial\Sigma}|\psi|^{\frac{2(n-1)}{n-2}}\,da\right)^{\frac{n-2}{n-1}},
\]
and hence
\[
2(\sigma_1(J)+b)\,|\partial\Sigma|^{\frac{1}{n-1}} \le Y(\psi).
\]
Taking the infimum over $\psi$ and then the supremum over conformal classes yields \eqref{305}. The equality case is obtained by tracing through the equalities exactly as in Theorem~\ref{202}.
\end{proof}

%---------------------------------------------------------
\subsection*{Acknowledgments}  
%---------------------------------------------------------
The authors would like to thank the anonymous referees for their careful reading of the manuscript and for their constructive comments and suggestions, which helped to improve the paper. 
We are also grateful to Abra\~ao Mendes for interesting discussions and suggestions regarding this work. 
The authors were partially supported by the National Council for Scientific and Technological Development -- CNPq (Brazil) and the Foundation for Research Support of the State of Alagoas -- FAPEAL. The first author was supported by CNPq (Grants 405468/2021-0, 308440/2021-8 and 402463/2023-9) and FAPEAL (Grant 60030.0000001758/2022). 
The second author was supported by CNPq (Grants 405468/2021-0 and 311136/2023-0) and FAPEAL (Grant 60030.0000000323/2023). 
This study was financed in part by the Coordena\c{c}\~ao de Aperfei\c{c}oamento de Pessoal de N\'ivel Superior -- Brasil (CAPES) -- Finance Code 001.

\bibliographystyle{amsplain}
\bibliography{bibliography}

\end{document}